\documentclass[11pt]{article}

\usepackage{fullpage}

\usepackage{enumerate}

\usepackage{amsmath}
\usepackage{amsthm}
\usepackage{amsfonts}
\usepackage{amssymb}
\usepackage{mathtools}

\usepackage{authblk}

\newtheorem{theorem}{Theorem}[section]
\newtheorem{definition}{Definition}[section]
\newtheorem{lemma}{Lemma}[section]
\newtheorem{remark}{Remark}[section]
\newtheorem{corollary}{Corollary}[section]
\newtheorem{example}{Example}[section]

\def \CV{\overline{\cov(\mes_G (V))}}
\def \CVm{\cov(\mes_G (V))}
\def \profile{\Psi}
\def \NEprofile{\tilde{\Psi}}
\def \dual{\Phi}
\def \to{\rightarrow}

\def \states{V}
\def \R {\mathbb{R}}
\def \N {\mathbb{N}}

\def \J {J}

\def \C {\mathcal{C}}
\def \A {\mathcal{A}}
\def \B {\mathcal{B}}
\def \Sc {\mathcal{S}}

\def \H {\textbf{H}}

\def \GM{(I, \lambda ,V,(A_i)_{i\in I} , \J )}
\def \ae {$\lambda$-almost every $i \in I$}

\def \y {\gamma}
\def \e {\eta}
\def \NEe {\tilde{\eta}}
\def \m {m}
\def \mes{\mathcal{P}}

\def \d{\mathrm{d}}
\def \cov{\mathrm{cov}}
\def \norm{\|}
\def \abs{|}

\title{Learning in anonymous nonatomic games \\ with applications to first-order mean field games}
\author[1]{Saeed Hadikhanloo\footnote{Universit\'e Paris-Dauphine, PSL Research University, CNRS, UMR [7534], CEREMADE, 75016 PARIS, FRANCE}}
\date{}

\begin{document}

\maketitle

\begin{abstract}
We introduce a model of anonymous games with the player dependent action sets. We propose several learning procedures based on the well-known Fictitious Play and Online Mirror Descent and prove their convergence to equilibrium under the classical monotonicity condition. Typical examples are first-order mean field games.
\end{abstract}

\tableofcontents

\section{Introduction}
Mean field games (MFGs) are symmetric differential games with an infinite number of non-atomic players. The model was first introduced simultaneously by Lasry, Lions (\cite{ LL06cr2},\cite{LL07mf}) and Huang, Caines, Malham\'e (\cite{HCMdec2003},\cite{HCMieeeAC06}). In the game, each player chooses a control and incurs a cost depending on their control and the evolving distribution of all the other players' states. More formally, a typical player chooses a path $ \y : [0 , T] \to \R^d , \; \y(0) = x$ via a control $ \d \y_t = \alpha_t \d t $
and incurs the cost:
$$\J ( \y , (m_t)_{t\in [0,T]} ) = \int_{0}^{T} \left( L(\y_t , \alpha_t) + f(\y_t , m_t) \right) \d t + g(\y_T , m_T)$$
where $(m_t)_{t\in [0,T]} \subseteq \mes(\R^d)$ is the evolving distribution of other players. The Lagrangian $L : \R^d \times \R^d \to \R$ captures the running cost depending on the velocity and $f,g : \R^d \times \mes(\R^d) \to \R$ are the couplings describing the interaction cost of the player with the distribution of the other players. The optimal control of a player can be obtained by solving the Hamilton-Jacobi equation:
$$-\partial_t u +H(x,\nabla u(t,x)) =f(x,m_t), \quad u(T,x)=g(x, m_T)$$
with $H(x,p)= - \inf_{v \in \R^d} \langle p , v \rangle + L(x,v)$. The desired optimal control will be computed as $$\tilde{\alpha}(t,x) = - D_p H(x,\nabla u(t,x)).$$
If every player chooses their optimal control, the evolving distribution of players is given by the Fokker-Planck equation: 
$$ \partial_t m - { \rm div} ( mD_p H (x,\nabla u)) =0, \quad m(0,x)=m_0(x).$$
Hence the notion of Nash Equilibrium (or stability) is captured by the system of coupled Hamilton-Jacobi (backward) and Fokker-Planck (forward) equations written above.

The equilibrium configuration in MFGs is quite complicated and its occurrence requires a huge amount of information and a large degree of cooperation between players. The question of formation of equilibrium arises naturally. Thus, one would guess that the MFG equilibrium is justifiable because there is a reasonable way of adapting (or learning) of players via observation and revision of the beliefs about the other players' behavior.

There are several learning procedures in games with finitely many players and/or a finite number of actions per player (see for example the monograph \cite{FL98}). Here we extend two of the most known of them to nonatomic games with continuous action sets: Fictitious Play (Brown \cite{BrownFictitiousPlay}) and Online Mirror Descent (Nemirovski, Yudin \cite{NemYud83}). In the current article, our main purpose is to prove the convergence of these procedures to Nash equilibrium in first-order MFGs; however, since the approach can be used for a larger class of games, we work under a general framework.

Anonymous games model the conflict situations where the dependency of the costs of the players to the action of their adversaries, are through the distribution of their chosen actions. The \textit{population game} is a class of anonymous games with a set of of non-atomic players who choose among a finite number of actions (see for example \cite{HS2009}). Mas-Collel \cite{Mas1984} proposed a type of anonymous games with a continuum of players where actions are chosen from an identical set and the cost functions depend on the players' types. Our approach here is different; in the sense that  actions may be chosen from different sets, but the cost functions are identical. For example this is the case in first-order MFGs; the players choose the paths with fixed (player dependent) initial positions as their actions, and the cost function is identical for all players.

In nonatomic games, Nash equilibria are defined up to a zero-measure set of players, i.e. the infinitesimal subsets which has no effect on the whole result of the game. Similar to the previous works (Schmeidler \cite{Schmeidler}, Mas-Collel \cite{Mas1984}) we employ a fixed point theorem to infer the existence of equilibrium under continuity and compactness.

Here we work under a well-known monotonicity condition introduced in game theory by Rosen \cite{Ros65}. The strict monotonicity yields the uniqueness of the Nash equilibrium (Haufbauer, Sandholm \cite{HS2009}, Blanchet, Carlier \cite{BC2014}). A similar definition appears in MFG (Lasry, Lions \cite{ LL06cr2},\cite{LL07mf}) which also implies the uniqueness of equilibrium. In anonymous games with (not necessarily strict) monotone costs, equilibrium uniqueness is a direct consequence of monotonicity and the "unique minimiser" condition. 

Haufbauer, Sandholm \cite{HS2009} introduced the \textit{stable games} as the population games with monotone costs. They prove the convergence of many learning dynamics including Best Response Dynamic, Replicator Dynamic, ... to the unique equilibrium. Their techniques have inspired our approach of the convergence results for Fictitious Play (Section \ref{section:Fictitious Play in Anonymous Games}) which is nothing but a discrete version of the best response dynamic.

As introduced by Brown\cite{BrownFictitiousPlay} and Robinson\cite{RO51}, fictitious play describes a learning procedure in which a fixed game is played over and over in repeated discrete rounds. At every round, each player sets their belief as the empirical frequency of play of the player's opponents, and then chooses its best action with respect to this belief. Convergence to a Nash equilibrium has been proved for different classes of finite games, for example potential games (Monderer, Shapley \cite{MondrerShapley2}), zero sum games (Robinson \cite{RO51}) and $2 \times 2$ games (Miyasawa \cite{Mi61}). Cardaliaguet, Hadikhanloo \cite{CDHD2015} proved the convergence of fictitious play in first and second order potential MFGs. Our approach here covers a different class of first-order MFGs, i.e. the ones with monotone costs.

The second procedure we consider is Online Mirror Descent (OMD). The method was first introduced by Nemirovski, Yudin \cite{NemYud83}, as a generalization of standard gradient descent. The form of the algorithm is closely related to the notion of No-Regret procedures in Online Optimization. A good explanatory introduction can be found in Shalev Shwartz\cite{Shwartz2011}. Roughly speaking, the procedure deals with two variables, a primal one and a dual one. They are revised at every round; the dual is revised by using the sub-gradient of the objective function and the primal is obtained by a \textit{quasi} projection via a strongly convex penalty function on the convex domain. Mertikopolous \cite{Mert2016} proved the convergence of OMD in a large class of games with convex action sets. Here we examine the convergence properties of OMD in monotone anonymous games with a possibly infinite number of players.

In the proof of convergence of both procedures to the Nash equilibrium, we define a value $ \phi_n \in \R, n \in \N$ measuring how much the actual behavior at step $n$ is far from being an equilibrium; in Fictitious Play the quantity $\phi_n$ is defined by the best response function and in OMD by using the Fenchel coupling. Then we prove $\lim_{n \to \infty} \phi_n = 0$ that gives our desired convergence toward the equilibrium.

Here is how the paper is organized: in Section 2 a general model of anonymous game is proposed. The notion of Nash equilibrium is reviewed and the existence is proved under general continuity conditions. Then we introduce the definition of monotonicity in terms of the cost function, and its consequence on the uniqueness of the Nash equilibrium. Section 3 is devoted to the definition of fictitious play and its convergence under Lipschitz conditions. Section 4 deals with the online mirror descent algorithm and its convergence. Section 5 shows that the first-order MFG can be considered as an example of anonymous games and shows that the previous results can be applied under suitable conditions. For sake of completeness, we provide in the Appendix some disintegration theorems which are used in the proofs.

\section{Anonymous Games}

\subsection{Model}
Let us introduce our general model of anonymous game $G$. For a measure space $X$ let $\mes(X)$ denotes the set of probability measures on $X$. Let $I$ be the set of players and $\lambda \in \mes(I)$ a prior non-atomic probability measure on $I$ modeling the repartition of players on $I$. Let $V$ be a measure space. For every player $i\in I$, let $A_i \subset \states$ be the action set of $i$. Define the set of admissible profiles of actions
$$ \A = \{ \profile : I \to V \; {\rm measurable} \; \abs \; \profile (i) \in A_i \; \text{for $\lambda$-almost every} \; i \in I \}.$$
We identify the action profiles up to $\lambda-$zero measure subsets of $I$, i.e. $\profile_1 = \profile_2$ iff $\profile_1 (i) = \profile_2(i)$ for \ae. The induced measure of a typical profile $\Psi\in \A$ on the set of actions, that captures the portion of players who have chosen a given subset of actions, is denoted by $ \profile \sharp \lambda  \in \mes(V)$. More precisely, $\profile\sharp \lambda$ is the push-forward of measure of $\lambda$ by application $\profile$, that is for every measurable set $ B \subseteq V$ we have $\profile\sharp \lambda(B) = \lambda(\profile^{-1}(B))$. Since the set of $ \profile \sharp \lambda$ for all admissible profiles $\profile$, may be different from $\mes(V)$, it is sufficient to work with:
$$\mes_G (V) = \{ \; \eta \in \mes (V) \; \abs \; \exists \; \profile \in \A: \; \eta = \profile \sharp \lambda \; \}.$$
For every $i\in I$ let $c_i : \A \to \R$ be the cost payed by player $i$. We call the game anonymous, if for every player $i\in I$, there exists $\J_i : A_i \times \mes_G (V) \to \R$ such that $c_i (\profile) = \J_i (\profile(i) , \profile \sharp \lambda)$. In other words, $\J_i (a,\eta)$ captures the cost endured by a typical player $i\in I$, whose action is $a\in A_i$ while facing the distribution of actions $\eta \in \mes(V)$ chosen by other players. We consider here anonymous games where the players have identical cost function, i.e. there is $\J : V \times \mes_G (V) \to \R$ such that for every $i\in I$ we have $\J_i = \J$. We use the following notation for referring to such game: $$G=\GM.$$

\begin{example}[Population Game \cite{HS2009}]
Set $I = [0,1]$ be the set of players and $\lambda$ the Lebesgue measure as the distribution of players on $I$. Let $N\in \N$ represents the number of populations in the game i.e. there is a partition of players $I_1,I_2, \cdots , I_N \subseteq I$ where for every $1 \leq p \leq N , I_p \subseteq I$ represents the set of players belonging to population $p$. For every player $i\in I$ suppose the set of actions $A_i$ is finite and depends only on the population where the player $i$ comes from, i.e. for every population $p$ there is $S_p$ such that for all $i \in I_p$ we have $A_i = S_p$. Set $V=\cup_{p} S_p$. For every population $p$ the cost function has the form $J_p : S_p \times \Delta(V) \to \R$ where $J_p(a,(m_j)_{1 \leq j \leq \abs V\abs})$ is the cost payed by a typical player in population $p$ whose action is $a\in S_p$ while facing $(m_j)_{1 \leq j \leq \abs V\abs}$ where  for every $1 \leq j \leq \abs V\abs$, $m_j \geq 0$ is the portion of players who have chosen action $j\in V$. The form of the cost function illustrates the fact that the population games are anonymous.
\end{example}

\begin{example}
	In section 5, we show that the First order MFG is an anonymous game with suitable actions sets and cost function.
\end{example}


\subsection{Nash Equilibria}
Inspired from the notion of Nash equilibrium in non-atomic games (see Schmeidler \cite{Schmeidler}, Mas-Collel \cite{Mas1984}), we omit the effect of zero measure subsets of players in the definition of equilibria:
\begin{definition}
a profile $\NEprofile \in \mathcal{A} $ is called a Nash equilibrium if 
$$\NEprofile(i) \in \arg \min_{a\in A_i} \J (a, \NEprofile \sharp \lambda ) \quad \text{for \ae}.$$
The corresponding distribution $\NEe = \NEprofile \sharp \lambda$ is called a Nash (or equilibrium) distribution. 
\end{definition}

One can note that the definition of Nash equilibrium highly depends on the prior distribution of players i.e. $\lambda$. The following Theorem gives a sufficient condition under which the game possesses at least one equilibrium. Let $I$ be a topological and $V$ be a metric space (with $\B(I) , \B(V)$ as their $\sigma-$fields). Suppose the $A_i$'s are uniformly bounded for $\lambda$-almost every $i \in I$, i.e. there exist $M >0, v\in V$ such that:  
\begin{equation}\label{BoundednessCondition}
	\text{for $\lambda-$almost every $i \in I$ and every $a\in A_i$}: \quad \d_V(v,a) < M.
\end{equation}
This condition gives us $\mes_G(V) \subseteq \mes_1(V)$ where:
$$\mes_1(V) = \{ \; \e \in \mes(V) \; \abs \; \exists v \in V: \; \int_{V} \d_V (v,a) \; \d \e (a) < +\infty \; \}$$
endowed with the metric:
$$\d_1 (\e_1 , \e_2) = \sup_{f: V \to \R , \; \text{1-Lipschitz}} \int_{V} f(a) \; \d (\e_1 - \e_2)(a).$$
For technical reasons we work with closure convex hull of $\mes_G(V)$ i.e. $\CV$.
\begin{definition}
We say $G=(I, \lambda ,V,(A_i)_{i\in I} , \J )$ satisfies the \textit{Unique Minimiser} condition, if for every $\e \in \overline{\cov(\mes_G (V))} $, there exists $I_\e \subseteq I $ with $\lambda(I \setminus I_\e) = 0$, such that for all $i\in I_\e$ there is exactly one $a\in A_i$ minimizing $\J (\cdot,\eta)$ in $A_i$. 
\end{definition}
Informally, the definition says facing to every distribution of actions, (almost) every player has a unique best response.
\begin{definition}
	A correspondence $A:I \to V, A(i)=A_i$ is called continuous if:
	\begin{itemize}
		\item
		it is upper semi continuous i.e. the graph $\{(i,a) \in I\times V \; \abs \; a\in A_i \}$ is closed in $I \times V$,
		\item
		it is lower semi continuous i.e. for every open set $U\subseteq V$ the set $\{ i \in I \; \abs \; A_i \cap U \neq \emptyset \}$ is open in $I$.
	\end{itemize}
\end{definition}
For more detailed theorems about set valued maps, see \cite{AubFra}.The following Theorem asserts sufficient conditions for existence of at least a Nash equilibrium:
\begin{theorem}\label{THM:NashExist}
Let $G=(I, \lambda ,V,(A_i)_{i\in I} , \J )$ be an anonymous game. Suppose the following conditions (\H) hold:
\begin{enumerate}[(i)]
	\item
	the correspondence $A : I \to V  , \; A(i)=A_i$ is continuous and compact valued,
	\item
	There is an extension $\J: V \times \CV \to \R$ which is lower semi-continuous,
	\item
	the function ${\rm Min} : I \times \mes_G (V) \to \R, \; {\rm Min} (i , \e) := \min_{a \in A_i} \J(a,\e)$ is continuous,
	\item
	$\CV$ is compact,
	\item
	$G$ satisfies the unique minimiser condition.
	\end{enumerate}
Then $G$ will admit at least a Nash equilibrium.
\end{theorem}
Assumptions $(i-iv)$ provide enough continuity and compactness conditions we need for the fixed point theorem. The assumption $(v)$ allows us to prove the existence of pure Nash equilibrium. In addition, it is crucial for the uniqueness of equilibrium and convergence results in learning procedures that we will propose. So we add it here as an assumption for being coherent in the entire article. Before we start the proof let us provide some lemmas which will be used here and in the rest of paper:
\begin{lemma} \label{LEMMA:ContBR}
Define the best response correspondence as follows
$$BR: I \times \overline{\cov(\mes_G (V))} \to V, \quad BR(i , \e) = \arg \min_{a \in A_i} \J (a,\e).$$
If the (\H) conditions hold, then for every $\e \in \mes_G(V)$ the correspondence $BR(\cdot , \e) : I \to V $, that is almost everywhere singleton, is almost everywhere continuous and hence measurable.
\end{lemma}
\begin{proof}
Fix $\e \in \overline{\cov(\mes_G (V))}$. According to the unique minimiser condition there exists $I_{\e} \subseteq I$ with $\lambda (I \setminus I_{\e}) = 0$, such that $BR(i,\e)$ is singleton for every $i\in I_{\e}$. We will show the continuity of the restricted best response function $BR(\cdot , \e) : I_{\e} \to V$ which completes our proof. Consider $i,i_n \in I_{\e}$ such that $i_n \to i$. Set $a_n = BR (i_n , \e)$. The set $\{ a_n \}_{n \in \N}$ is precompact since $A:I \to V$ is a compact valued correspondence and hence $A(\{i_n \}_{n\in \N} \cup \{ i \}) = \cup_{n} A_{i_n} \cup A_i$ is compact. Suppose $\tilde{a} \in V$ is an accumulation point of $\{ a_n \}_{n \in \N}$. So there is a sub-sequence $\{ a_{n_k} \}_{k \in \N}$ such that $\lim_{k\to \infty} a_{n_k}= \tilde{a}$. We have $\tilde{a} \in A_i$ since the correspondence $A:I \to V$ is upper semi continuous and $a_n \in A_{i_n}$. By definition $\J (a_n , \e) = {\rm Min} (i_n , \e)$ which gives:
$$  \J(\tilde{a},\e) \leq \liminf_{n_k} \J (a_{n_k} , \e) = \liminf_{n_k} \; {\rm Min} (i_{n_k} , \e) = {\rm Min} (i , \e),$$
since the Min function is continuous. It yields $\tilde{a} = BR(i,\e)$. So every accumulation point of $\{a_n\}_{n\in \N}$ should be $BR(i,\e)$ which shows $a_n \to BR(i,\e)$.
\end{proof}

\begin{lemma}\label{LEMMA:ContTheta}
Define the best response distribution function $\Theta : \overline{\cov(\mes_G (V))} \to \mes_G (V)$ as follows:
$$ \Theta(\e) = BR(\cdot , \e) \sharp \lambda, \quad \text{for every } \e\in \overline{\cov(\mes_G (V))}.$$
If the (\H) conditions hold then $\Theta$ is continuous.
\end{lemma}
\begin{proof}
Let $\e_n \to \e$. If $J = I_{\e} \cap_{n\in \N} I_{\e_n}$ then we have $\lambda(I \setminus J)=0$. One can show as for Lemma \ref{LEMMA:ContBR} that for every $i \in J$:
$$ BR(i,\e_n) \to BR(i,\e).$$
Since the $A_i$'s are uniformly bounded for $\lambda-$almost every $i \in J$, the dominated Lebesgue convergence Theorem implies $\int_{I} \d_V (BR(i,\e_n) , BR(i,\e) ) \; \d \lambda(i) \to 0$. Thus $\Theta(\e_n) \xrightarrow{\d_1} \Theta(\e)$ since:
$$\d_1(\Theta(\e_n) , \Theta(\e)) = \sup_{f: V \to \R , \; \text{1-Lipschitz}} \int_{V} f(v) \; \d (\Theta(\e_n) - \Theta(\e))(v) = $$ $$\sup_{f: V \to \R , \; \text{1-Lipschitz}} \int_{I}  \left( f(BR(i,\e_n)) - f(BR(i,\e)) \right) \; \d \lambda(i)\leq \int_{I} \d_V (BR(i,\e_n) , BR(i,\e) ) \; \d \lambda(i) \to 0.$$
\end{proof}

\begin{proof}[Proof of Theorem \ref{THM:NashExist}]

Consider the best response distribution function $\Theta$ defined in Lemma \ref{LEMMA:ContTheta}. We have by definition $$\Theta(\CV) \subset \mes_G(V) \subset \CV,$$ which implies that the image of $\Theta$ is precompact. Since $\Theta$ is continuous (Lemma \ref{LEMMA:ContTheta}) and $\overline{\cov(\mes_G (V))}$ is convex, by the Schauder's fixed point theorem, there is $\NEe \in \overline{\cov(\mes_G (V))}$ such that $\Theta(\NEe) = \NEe$. Since $\Theta(\NEe) = BR(\cdot , \NEe) \sharp \lambda \in \mes_G (V) $ so if we set $\NEprofile (\cdot)  = BR(\cdot , \NEe )\in \A$ then
$$\NEprofile \sharp \lambda = \NEe, \quad \NEprofile(i) \in \arg \min_{a\in A_i} \J (a, \NEe) \quad \text{for \ae}.$$
This means $\NEprofile$ is the desired Nash equilibrium.
\end{proof}

\subsection{Anonymous Games with Monotone Cost}

Here we give a definition of monotonicity and its additional consequences on the structure of the game and its equilibria. 
\begin{definition}
The anonymous game $G=\GM$ has a monotone cost $\J$ if for any $\e,\e' \in \CV$ one has:
$$\int_{V} \left| \J (a,\eta) \right|  \; \d \eta'(a) < +\infty,$$
and
$$
\int_{V} \left( \J (a,\eta) - \J (a,\eta') \right) \; \d (\eta - \eta')(a) \geq 0,
$$
and it is strict monotone if the later inequality holds strictly for $\e \neq \e'$.
\end{definition}
Intuitively, this condition describes the aversion of the players for choosing actions that are chosen by many of players i.e. congestion avoiding effect. In other words, on average the players dislike the crowded actions.
\begin{remark}\label{Col1}
If $\J$ is monotone and if $\NEprofile \in \mathcal{A}$ is a Nash equilibrium, then for every $\profile \in \mathcal{A}$ we have:
$$ \text{if } \; \NEe = \NEprofile \sharp \lambda \; , \; \e = \profile \sharp \lambda \; :
\quad  \int_{V} \J (a,\eta) \; \d (\e - \NEe) (a) \geq \int_{V} \J (a,\NEe) \; \d (\e - \NEe) (a) \geq 0.$$
\end{remark}
\begin{proof}
Since $\J$ is monotone we have $\int_{V} \left( \J (a,\e) - \J (a,\NEe) \right) \; \d (\e - \NEe) (a) \geq 0$ and so:
$$\int_{V} \J (a,\e) \; \d (\e - \NEe) (a) \geq
\int_{V} \J (a,\NEe) \; \d (\e - \NEe) (a) . $$
On the other hand
$$ \int_{V} \J (a,\NEe) \; \d (\e - \NEe) (a)
= \int_{I} \left( \J(\profile (i),\NEe) - \J(\NEprofile (i),\NEe) \right) \; \d \lambda (i) $$
by the definition of push-forward measures. Since $\NEprofile$ is an equilibrium, for \ae, we have $\J(\profile (i),\NEe) - \J(\NEprofile (i),\NEe) \geq 0$, which gives our result.
\end{proof}
The strict monotonicity yields the uniqueness of the Nash equilibrium in different frameworks, e.g. Haufbauer, Sandholm \cite{HS2009}, Blanchet, Carlier \cite{BC2014}, Lasry, Lions \cite{LL06cr2}. In the following we show that in Anonymous Games, the Monotonicity and Unique Minimiser condition are sufficient for the uniqueness of the equilibrium.
\begin{theorem}\label{NashUnique}
Consider a game $G=\GM$. Then the game $G$ admits at most one Nash equilibrium if $\J$ is monotone and $G$ satisfies the Unique Minimiser condition.

\end{theorem}
\begin{proof}
Let $\profile_1 , \profile_2 \in \mathcal{A}$ be two Nash equilibria. We will show that $\profile_1(i) = \profile_2 (i)$ for \ae . Set $\eta_i = \profile_i \sharp \lambda$ for $i=1,2$. Since $\profile_1$ is an equilibrium, we have:
$$
\int_{I} \left( \J(\profile_1 (i),\eta_1) - \J(\profile_2 (i),\eta_1) \right) \; \d \lambda (i) \leq 0,
$$
since $\J (\profile_1 (i),\eta_1) \leq \J (\profile_2 (i),\eta_1)$ for $\lambda$-almost every $i\in I$. On the other hand:
$$
\int_{I} \left( \J (\profile_1 (i),\eta_1) - \J (\profile_2 (i),\eta_1) \right) \; \d \lambda (i) =
\int_{V} \J (a,\eta_1) \; \d (\eta_1 - \eta_2) (a),
$$
from the definition since $\profile_i \sharp \lambda = \eta_i$ for $i=1,2$. So $$\int_{V} \J(a,\eta_1) \; \d (\eta_1 - \eta_2) (a) \leq 0 \quad  \text{and (similarly)} \quad \int_{V} \J (a,\eta_2) \; \d (\eta_2 - \eta_1) (a) \leq 0,$$ which gives:
$$\int_{V} \left( \J(a,\eta_1) - \J (a,\eta_2) \right)  \; \d (\eta_1 - \eta_2) (a) \leq 0.$$
Hence by monotonicity of $\J $ we should have the equality in the later inequalities. So for \ae,  one has $\J (\profile_1 (i),\eta_1) = \J (\profile_2 (i),\eta_1)$ which gives our result since $\profile_1(i) \in A_i$ is the unique minimisers of $\J (\cdot , \eta_1)$ on $A_i$ so $\profile_1(i) = \profile_2(i)$ for \ae.
\end{proof}

\begin{remark}
One can similarly show that if $\J$ is strictly monotone and not necessarily satisfies the unique minimizer condition, then there exists at most one Nash equilibrium distribution.
\end{remark}

\section{Fictitious Play in Anonymous Games}\label{section:Fictitious Play in Anonymous Games}
Here we introduce a learning procedure similar to the Fictitious Play defined by Brown \cite{BrownFictitiousPlay} and prove its convergence to the unique Nash equilibrium when the game is monotone. 

Let $G=\GM$. For technical reasons, we suppose that conditions (\H) hold throughout this section. Suppose $G$ is being played repeatedly on discrete rounds $n=1,2,\ldots$. At  every round, the players set their belief equals to the average of the action distribution observed in the previous rounds and then react their best to such belief. At the end of the round players revise their beliefs by a new observation. More formally, consider $\profile_1 \in \A, \; \bar{\e}_1 = \e_1 = \profile_1 \sharp \lambda \in \mes(V)$ an arbitrary initial belief. Construct recursively $(\profile_n , \e_n , \bar{\e}_n) \in \A \times \mes(V) \times \mes(V)$  for $n=1,2,\ldots$ as follows:
\begin{equation}\label{FicitiousPlay}
\begin{array}{lrll}
(i)&\profile_{n+1}(i) &= &BR(i,\bar{\e}_{n}), \quad \text{for \ae}, \\
(ii)&\e_{n+1} &= &\profile_{n+1} \sharp \lambda ,\\
(iii)&\bar{\e}_{n+1} &= & \frac{n}{n+1} \bar{\e}_{n} + \frac{1}{n+1} \e_{n+1}.
\end{array}
\end{equation}
One should notice that by assumption $(\H)(v)$ and Lemma \ref{LEMMA:ContBR} the expressions in $(i,ii)$ are well defined. We will show now that this procedure converges to the Nash Equilibrium when $G$ is monotone.
\begin{theorem}\label{FictitiousPlayMonotone}
Consider an Anonymous game $G=(I, \lambda ,V,(A_i)_{i\in I} , \J )$ with a \textbf{monotone} Cost.
Suppose that exists $C>0$ such that for all $a,b \in V , \; \e,\e' \in \CV$:
\begin{equation}\label{LPCONDFP}
\begin{split}
\abs \J(a,\e) - \J(a,\e') - \J(b,\e) + \J(b,\e') \abs &\leq C \; \d_V(a,b) \; \d_1 (\e , \e'), \\
\abs \J(a,\e) - \J(a,\e') \abs &\leq C \; \d_1 (\e , \e').
\end{split}
\end{equation}
Construct $(\profile_n , \e_n , \bar{\e}_n) \in \A \times \mes(V) \times \mes(V)$ for $n \in \N$ by applying the fictitious play procedure proposed in \eqref{FicitiousPlay}. 
Then:
$$\e_n , \bar{\e}_n \xrightarrow{\d_1} \NEe$$
where $\NEe \in \mes_G (V)$ is the unique Nash equilibrium distribution.
\end{theorem}

Inspired from \cite{HS2009}, the proof requires several steps. The key idea is to use the quantity $\phi_n \in \R $ defined by
$$ \phi_n =\int_V \J (a,\bar{\e}_n) \; \d (\bar{\e}_n - \e_{n+1})(a) , \quad \text{for every } n \in \N.$$
Since the best response distribution of $\bar{\e}_n$ is $\e_{n+1}$, the quantity $\phi_n$ describes how much $\bar{\e}_n$ is far from being an equilibrium. By using monotonicity and the regularity conditions, one gets
$$\forall n\in \N: \quad \phi_{n+1} - \phi_n \leq -\frac{1}{n+1} \phi_n + \frac{\epsilon_n}{n},$$
for suitable $\{\epsilon_n \}_{n\in \N}$ such that $\lim_{n\to \infty} \epsilon_n = 0$. We show the later inequality is sufficient to prove $\lim_{n\to \infty} \phi_n = 0$ and then we conclude that the accumulation points of $\bar{\e}_n , \e_n$ is the equilibrium distribution $\NEe$. As one will see, the unique minimiser assumption plays a key role in Lemma \ref{LEMFICT2} and hence in our main result.

\begin{lemma}\label{LEMFICT}
Consider a sequence of real numbers $\{ \phi_n \}_{n \in \N}$ such that $\liminf_n \phi_n \geq 0$. If there exists a real sequence $\{ \epsilon_n \}_{n\in \N}$ such that $\lim_{n \to \infty} \epsilon_n = 0$ and :
$$\forall \; n\in \N: \quad \phi_{n+1} - \phi_n \leq -\frac{1}{n+1} \phi_n + \frac{\epsilon_n}{n},$$
then $\lim_{n \to \infty} \phi_n = 0$.
\end{lemma}
\begin{proof}
Let $b_n = n \phi_n$ for every $n \in \N$. We have:
$$\forall \; n\in \N: \quad \frac{b_{n+1}}{n+1} - \frac{b_{n}}{n} \leq - \frac{b_{n}}{n(n+1)} + \frac{\epsilon_n}{n},$$
which implies $b_{n+1} \leq b_n + (n+1) \epsilon_n /n \leq b_n + 2 \abs \epsilon_n \abs$. Then we get $b_n \leq b_1 + 2\sum_{i=1}^{n-1} \abs \epsilon_i \abs $ for $n\in \N$ and so:
$$0 \leq \liminf_n \phi_n \leq \limsup_n \phi_n \leq \limsup_n \frac{b_1 + 2 \sum_{i=1}^{n-1} \abs \epsilon_i \abs }{n} = 0.$$
which proves $\lim_{n \to \infty} \phi_n = 0$.
\end{proof}
\begin{lemma}\label{LEMFICT2}
Let $(\e_n)_{n \in \N}$ be defined by \eqref{FicitiousPlay}. Then
$$\d_1 (\bar{\e}_{n} , \bar{\e}_{n+1}) = O(1/n), \quad \lim_{n\to \infty} \d_1 (\e_n , \e_{n+1}) = 0.$$
\end{lemma}
\begin{proof}
Let $M>0, v\in V$ be chosen from \eqref{BoundednessCondition}. For every 1-Lipschitz continuous map $f:V \to \R$ we have:
$$\left| \int_V f(a) \; \d (\bar{\e}_{n+1} - \bar{\e}_{n}) \right| = \frac{1}{n+1} \left| \int_V f(a) \; \d (\e_{n+1} - \bar{\e}_{n})(a) \right| =$$ $$ \frac{1}{n+1} \left| \int_V (f(a) - f(v)) \; \d (\e_{n+1} - \bar{\e}_{n})(a) \right| \leq \frac{1}{n+1}  \left( \int_V \d_V (a,v) \; \d \e_{n+1}(a) + \frac{1}{n} \sum_{k=1}^{n} \int_V \d_V (a,v) \; \d \e_{k} (a)  \right).$$
By the definition we have:
$$\int_V \d_V (a,v) \; \d \e_{k} (a) = \int_V  \d_V (\profile_k(i),v)  \; \d \lambda (i) \leq M, \quad \text{for every }k\in \N.$$
So we can write 
$$\left| \int_V f(a) \; \d (\bar{\e}_{n+1} - \bar{\e}_{n}) \right| \leq \frac{2M}{n+1},$$
or $\d_1 (\bar{\e}_{n} , \bar{\e}_{n+1}) \leq \frac{2M}{n+1}$ since $f$ is arbitrary.

For the second part of the lemma, let us consider the best reply distribution function $\Theta$ defined in Lemma \ref{LEMMA:ContTheta}. Since $\Theta$ is continuous (Lemma \ref{LEMMA:ContTheta}) and $\overline{\cov(\mes_G (V))}$ is compact, there exists a non decreasing continuity modulus
$$\omega : \R_+ \to \R_+ , \quad \lim_{x \to 0^+} \omega(x) = 0$$
such that:
$$\forall \; \e_1,\e_2 \in \overline{\cov(\mes_G (V))} : \quad \d_1 (\Theta(\e_1) , \Theta(\e_2)) \leq \omega (\d_1 (\e_1 , \e_2)).$$
Since for all $n \in \N$ we have $\bar{\e}_{n} \in \overline{\cov(\mes_G (V))}$ and $\Theta(\bar{\e}_n) = \e_{n+1}$ we have
$$  0 \leq \d_1 (\e_{n+1} , \e_{n+2}) = \d_1 (\Theta(\bar{\e}_{n}) , \Theta(\bar{\e}_{n+1})) \leq \omega ( \d_1 (\bar{\e}_{n} , \bar{\e}_{n+1})).$$
It gives our desired result since $\d_1 (\bar{\e}_{n} , \bar{\e}_{n+1}) = O(1/n)$.
\end{proof}
The proof of previous lemma relies heavily on the unique minimizer assumption. Instead without it, one cannot conclude that $\e_n , \e_{n+1}$ are close even if $\bar{\e}_n , \bar{\e}_{n+1}$ are so. Even for $\bar{\e}_n = \bar{\e}_{n+1}$, one might have very different best responses $\e_n $ and $ \e_{n+1}$.
\begin{proof}[Proof of Theorem \ref{FictitiousPlayMonotone}]
Let $ \{\phi_n \}_{n\in \N}$ be defined by:
$$ \phi_n =\int_V \J (a,\bar{\e}_n) \; \d (\bar{\e}_n - \e_{n+1})(a) , \quad \text{for every } n \in \N.$$
We have $\phi_n \geq 0$ for all $n\in \N$. Indeed, rewriting the definition of $\phi_n$, we have:
$$ \phi_n =\int_I \frac{1}{n}\sum_{j=1}^{n} \left( \J (\profile_{j}(i),\bar{\e}_n) - \J (BR(i,\bar{\e}_n),\bar{\e}_n) \right) \; \d \lambda (i),$$
and the positiveness comes from the definition of the best response. We now prove that exists $C>0$ such that:
\begin{equation}\label{InEq:PhiFict}
\phi_{n+1} - \phi_{n} \leq - \frac{1}{n+1} \phi_{n} + C\frac{\d_1 (\e_{n} , \e_{n+1}) + 1/n}{n}, \quad \text{for every} \; n \in \N.
\end{equation}
Let us rewrite $\phi_{n+1} - \phi_{n} = A + B$, where:
$$A= \int_V \J (a,\bar{\e}_{n+1}) \; \d \bar{\e}_{n+1} (a) - \int_V \J (a,\bar{\e}_n) \; \d \bar{\e}_n (a), $$ 
$$
B= \int_{V}  \J(a,\bar{\e}_{n}) \; \d \e_{n+1}(a) - \int_{V}  \J(a,\bar{\e}_{n+1}) \; \d \e_{n+2}(a).$$
We have:
\begin{equation*}
\begin{split}
B &\leq \int_{V}  \J(a,\bar{\e}_{n}) \; \d \e_{n+2}(a) - \int_{V}  \J(a,\bar{\e}_{n+1}) \; \d \e_{n+2}(a) \\
&= \int_{V} ( \J(a,\bar{\e}_{n}) - \J(a,\bar{\e}_{n+1}) ) \; \d \e_{n+2}(a) \\
&\leq \int_{V} ( \J(a,\bar{\e}_{n}) - \J(a,\bar{\e}_{n+1}) ) \; \d \e_{n+1}(a) + \frac{C}{n} \d_1 (\e_{n+1} , \e_{n+2}),
\end{split}
\end{equation*}
since by $\eqref{LPCONDFP}$ and Lemma \ref{LEMFICT2} there exists $C$ such that the function $\J(\cdot ,\bar{\e}_{n}) - \J( \cdot ,\bar{\e}_{n+1}) : V \to \R$ is a $C/n-$Lipschitz continuous function. Let us rewrite the expression $A$ as follows:
\begin{equation*}
\begin{split}
A &= 
\int_V \J (a,\bar{\e}_{n+1}) \; \d (\bar{\e}_{n} + \frac{1}{n+1}(\e_{n+1} - \bar{\e}_n) ) (a) - \int_V \J (a,\bar{\e}_n) \; \d \bar{\e}_n (a) \\
&=
\int_V (\J (a,\bar{\e}_{n+1}) - \J (a,\bar{\e}_n) ) \; \d \bar{\e}_{n} (a) + \frac{1}{n+1} \int_V \J (a,\bar{\e}_{n+1}) \; \d (\e_{n+1} - \bar{\e}_n) ) (a) \\
&\leq
\int_V (\J (a,\bar{\e}_{n+1}) - \J (a,\bar{\e}_n) ) \; \d \bar{\e}_{n} (a) + \frac{1}{n+1} \int_V \J (a,\bar{\e}_{n}) \; \d (\e_{n+1} - \bar{\e}_n) ) (a) + \frac{C}{n^2}
\end{split}
\end{equation*}
since by \eqref{LPCONDFP} and Lemma  \ref{LEMFICT2} we have  $\abs \J(a,\bar{\e}_n) - \J(a,\bar{\e}_{n+1}) \abs \leq C \; \d_1 (\bar{\e}_{n+1} , \bar{\e}_n) = O(1/n)$. So
$$
A \leq \int_V (\J (a,\bar{\e}_{n+1}) - \J (a,\bar{\e}_n) ) \; \d \bar{\e}_{n} (a) - \frac{\phi_n}{n+1} + \frac{C}{n^2}.$$
Then if we set $\epsilon_n = C (\d_1 (\e_{n+1} , \e_{n+2}) + 1/n )$, by using the above inequalities for $A,B$, we have :
\begin{equation}
\begin{split}
A+ B &\leq \int_V (\J (a,\bar{\e}_{n+1}) - \J (a,\bar{\e}_n) ) \; \d (\bar{\e}_{n} - \e_{n+1}) (a) - \frac{\phi_n}{n+1} + \frac{\epsilon_n}{n} \\
&= -(n+1) \int_V (\J (a,\bar{\e}_{n+1}) - \J (a,\bar{\e}_n) ) \; \d (\bar{\e}_{n+1} - \bar{\e}_{n}) (a) - \frac{\phi_n}{n+1} + \frac{\epsilon_n}{n} \\
& \leq - \frac{\phi_n}{n+1} + \frac{\epsilon_n}{n},
\end{split}	
\end{equation}
and the last inequality comes from the monotonicity assumption. By Lemmas \ref{LEMFICT}, \ref{LEMFICT2} inequality \eqref{InEq:PhiFict} implies $\phi_{n} \to 0$. Let $(\e , \bar{\e}) \in \overline{\mes_G(V)} \times \overline{\cov(\mes_G (V))} $ be an accumulation point of the set $\{ (\e_{n+1} , \bar{\e}_n ) \}_{n\in \N}$. We have $\e = \Theta(\bar{\e})$ due to the continuity of best response distribution function $\Theta$ (Lemma \ref{LEMMA:ContTheta}) and the fact that $\e_{n+1} = \Theta(\bar{\e}_n)$.

Take an arbitrary $\theta \in \mes_G(V)$. Since $\J$ is lower semi-continuous we have (see \cite{AGS} section 5.1.1):
$$\int_V \J (a,\bar{\e}) \; \d ( \bar{\e} - \theta ) (a) \leq \liminf \int_V \J (a,\bar{\e}) \; \d ( \bar{\e}_n - \theta ) (a)  = 
\liminf \int_V \J (a,\bar{\e}_n) \; \d ( \bar{\e}_n - \theta ) (a)
$$ 
$$=\lim \inf \int_V \J (a,\bar{\e}_n) \; \d ( \e_{n+1} - \theta )(a) + \phi_n \leq \lim \inf \phi_n =0 $$
since $\e_{n+1} = \Theta(\bar{\e}_n)$ and $\int_V \J (a,\bar{\e}_n) \; \d ( \e_{n+1} - \theta )(a) \leq 0$ for every $\theta \in \mes_G (V)$. So:
\begin{equation}\label{123}
	\forall \; \theta \in \mes_G (V): \quad \int_V \J (a,\bar{\e}) \; \d ( \bar{\e} - \theta ) (a) \leq 0.
\end{equation}
We rewrite the above inequality as follows: since $\bar{\e} \in \CV$ by Corollary \ref{COR:DisintCOV} we can disintegrate it with respect to $(A_i)_{i\in I}$ i.e. there are $\{ \bar{\e}^i \}_{i \in I} \subseteq \mes(V)$ such that for $\lambda-$almost every $i\in I$ we have $\text{supp}(\bar{\e}^i) \subset A_i$ and for every integrable function $f:V \to \R$:
$$ \int_I \int_{A_i} f(a) \; \d ( \bar{\e}^i)(a) \; \d \lambda(i) = \int_V f(a) \; \d \bar{\e} (a).$$
Specially for $f = \J (\cdot , \e)$ we have $\int_I \int_{A_i} \J (a,\bar{\e}) \; \d ( \bar{\e}^i)(a) \; \d \lambda(i) = \int_V \J (a,\bar{\e}) \; \d \bar{\e} (a),$
and for all $\profile \in \A$:
$$\int_I \int_{A_i}  J(\profile(i),\bar{\e}) \; \d ( \bar{\e}^i)(a) \; \d \lambda(i) = \int_V \J (a,\bar{\e}) \; \d ( \profile\sharp \lambda) (a).$$
Combining the previouse equailities with \eqref{123}, gives us:
$$\forall \profile \in \A: \quad \int_I \int_{A_i} ( \J (a,\bar{\e}) - J(\profile(i),\bar{\e}) ) \; \d ( \bar{\e}^i)(a) \; \d \lambda(i) = \int_V \J (a,\bar{\e}) \; \d ( \bar{\e} - \profile\sharp \lambda) (a) \leq 0.$$
In particular if $\profile = BR (\cdot , \bar{\e})$ we have:
$$ \int_I \int_{A_i} ( \J (a,\bar{\e}) - J(BR(i , \bar{\e}),\bar{\e}) ) \; \d ( \bar{\e}^i)(a) \; \d \lambda(i) \leq 0,$$
which gives the equality by definition of best response action. So by unique minimizer we have $\bar{\e}^i = \delta_{BR(i , \bar{\e})}$ for $\lambda-$almost every $i\in I$. It means $\bar{\e} = BR( \cdot , \bar{\e}) \sharp \lambda$ or $\bar{\e} = \Theta(\bar{\e})$. Hence $\bar{\e} = \e$ and they are both equal to $\NEe \in \mes_G (V)$, the unique fixed point of $\Theta$, or equivalently, the equilibrium distribution. 
\end{proof}

\section{Online Mirror Descent}

Here we investigate the convergence results by applying Online Mirror Descent (OMD) in anonymous games. The form of OMD algorithm is closely related to the online optimization and no regret notions. The reader can find a good explanatory note in \cite{Shwartz2011}. The goal of the algorithm is to act optimally in online manner by somehow "minimising" a function that itself changes at each step. In the game framework it is the cost function; it changes due to change of the acions chosen by adverseries in each round. As one can notice in the following, we need the structure of vector space for the action sets.

\subsection{Prelimineries}
Before we propose the main OMD, let us review some definitions and lemmas. 

\begin{definition}
Let $(W, \norm \cdot \norm_W)$ be a normed vector space. For $K>0$ we say that $h : W \to \R$ is a $K-$strongly convex function if 
$$\forall a_1,a_2 \in W, \; \forall \lambda \in [0,1] : \quad h(\lambda a_1 + (1-\lambda) a_2) \leq \lambda h(a_1) + (1-\lambda) h(a_2) - K\lambda(1-\lambda) \norm a_1-a_2 \norm_W ^2.$$
\end{definition}
\begin{definition}
The Fenchel conjugate of a function $h : W \to \R$ on a set $A \subseteq W $ is defined by:
$$ h^*_A : W^* \to \R: \quad h^*_A(y) = \sup_{a\in A} \; \langle y, a \rangle - h(a), \quad \text{for all } y\in W^*$$
and the related maximiser correspondence by:
$$ Q_A : W^* \to A: \quad Q_A(y) = \arg \max_{a\in A} \; \langle y, a \rangle - h(a), \quad \text{for all } y\in W^*.$$
\end{definition}
One can notice that $Q_A$ is not empty if $A$ is weakly closed and $h$ is weakly lower semi-continuous and coercive, i.e.
$$\lim_{a \to \infty} h(a)/\norm a \norm_W = +\infty.$$
If $W$ be a Hilbert space (so $W^* = W$) and $h(a)=\frac{1}{2} \norm a \norm_W^2$ then the correspondence $Q_A$ will be the classical projection on $A$:
$$Q_A(y) = \arg \max_{a\in A} \; \langle y, a \rangle_W - \frac{1}{2} \norm a \norm_W^2 = \arg \max_{a\in A} -\norm y - a \norm_W^2 = \pi_A(y).$$
\begin{lemma}\label{QLipschitz}
Let $h:W \to \R$ be a $K-$strongly convex function and $A$ a convex subset of $W$. For any $y_1,y_2 \in W^*$ let $a_i \in Q_A (y_i), \; i=1,2$. Then we have: $$2K \norm a_1 - a_2 \norm^2 _W \leq \langle y_1 - y_2 , a_1 - a_2 \rangle. $$
It implies $\norm a_1 - a_2 \norm_W \leq \frac{1}{2K} \norm y_1 - y_2 \norm_{W^*}$. In particular if $y_1=y_2$ then $a_1 = a_2$ i.e. the correspondence $Q_A (y)$ is either empty or single valued for every $y \in W^*$.
\end{lemma}
\begin{proof}
Since $A$ is convex, for every $\epsilon \in (0,1]$ we have $(1-\epsilon)a_1 + \epsilon a_2 \in A$. By definition:
$$\langle y_1 , a_1 \rangle - h(a_1) \geq \langle y_1 , (1-\epsilon)a_1 + \epsilon a_2 \rangle - h((1-\epsilon)a_1 + \epsilon a_2),$$
$$h((1-\epsilon)a_1 + \epsilon a_2) \leq (1-\epsilon)h(a_1) + \epsilon h(a_2) - K\epsilon(1-\epsilon) \norm a_1 - a_2 \norm^2.$$
So by combining the above inequalities:
$$\langle y_1 , a_1 \rangle - h(a_1) \geq \langle y_1 , (1-\epsilon)a_1 + \epsilon a_2 \rangle - (1-\epsilon)h(a_1) - \epsilon h(a_2) + K\epsilon(1-\epsilon) \norm a_1 - a_2 \norm^2,$$
which gives:
$$\epsilon \langle y_1 , a_1 - a_2 \rangle \geq \epsilon h(a_1) -\epsilon h(a_2) + K\epsilon(1 - \epsilon) \norm a_1 - a_2 \norm^2.$$
After dividing the both sides by $\epsilon$ and then $\epsilon \to 0^+$ we will get:
$$\langle y_1 , a_1 - a_2 \rangle \geq h(a_1) - h(a_2) + K \norm a_1 - a_2 \norm^2.$$
By exchanging the role of $(a_1 , y_1)$ and $(a_2 , y_2)$ we have:
$$\langle y_2 , a_2 - a_1 \rangle \geq h(a_2) - h(a_1) + K \norm a_2 - a_1 \norm^2.$$
It yields the desired result if we sum up the two later inequalities.
\end{proof}
\begin{definition}
Let $F: W \to \R$ be a convex function. We say that $ v \in W^*$ is a subgradient of $F$ at $a \in W$ if:
$$\forall \; b\in W : \quad F(b) - F(a) \geq \langle v, b-a \rangle,$$ 
and set $\partial F (a) \subseteq W^*$ the set of all subgradients at $a$.
\end{definition}
One can notice that if $F:W \to \R$ is differentiable at $a\in W$ then $\partial F (a) = \{ D F (a) \}$.
\subsection{OMD algorithm and convergence result}
Consider an anonymous game $G=(I, \lambda ,V,(A_i)_{i\in I} , \J)$. Suppose that the following conditions hold:
\begin{itemize}
	\item
	there is a normed vector space $(W, \norm \cdot \norm_W)$ such that $$ \bigcup_{i\in I} A_i \subseteq W \subseteq V,$$ and let $h:W \to \R$ be a $K-$strongly convex function for a real $K>0$.
	\item
	for every $i\in I$ the action sets $A_i$ are weakly closed in $W$ and $h$ is weakly lower semi-continuous and coercive (and hence $Q_{A_i}$ is single valued),
	\item
	for every $(a,\eta) \in W \times \mes_G (V)$ the function $\J (\cdot , \eta) : W \to \R$ is convex and exists a subgradient $y(a,\eta) \in \partial_a \J(\cdot , \e) \subseteq W^*$,
\end{itemize}
Let $\{\beta_n\}_{n\in \N}$ be a sequence of real positive numbers. Set an arbitrary initial measurable functions $\profile_0 \in \A , \; \eta_{0}=\profile_{0} \sharp \lambda , \; \dual_0 : I \to W^*$. The following procedure \eqref{OMD} is called the \textit{Online Mirror Descent (OMD)} on anonymous game $G$:
\begin{equation}\label{OMD}
\begin{array}{lllll}
(i)&\dual_{n+1}(i)&= &\dual_n (i) - \beta_n y (\profile_n (i) , \e_n), &\text{for every} \; i \in I \\
(ii)&\profile_{n+1}(i) &= &Q_{A_{i}} ( \dual_{n+1}(i) ) , &\text{for every} \; i \in I \\
(iii)&\e_{n+1} &= &\profile_{n+1} \sharp \lambda.&
\end{array}
\end{equation}

\begin{theorem}\label{OMDNAG}
Suppose one applies the OMD algorithm proposed in \eqref{OMD} for $\beta_n = \frac{1}{n}$. Suppose the following conditions hold:
\begin{enumerate}[i.]
	\item
	the game $G$ satisfies (\H) conditions,
	\item
	for every $i\in I$ the action sets $A_i$ are convex and 
	and exists $M >0$ such that for $\lambda-$almost every $i\in I$ we have $\norm a \norm_W \leq M$ for all $a\in A_i$ 	and we have $R(M) := \sup_{\norm a \norm \leq M} \abs h(a) \abs < +\infty,$
	\item
	the map $\dual_0 : I \to W^*$ is bounded,
	\item
	the cost function $\J$ is monotone,
	\item
	there exists $k>0$ such that:
	\begin{equation}\label{OMDcondition}
	\forall n\in \N: \quad \norm y (\profile_{n}(i) , \e_n) \norm_{W^*} \;  , \; \norm y (\profile_{n}(i) , \NEe) \norm_{W^*} \leq k, \quad \text{for $\lambda-$almost every }i\in I,
	\end{equation}
\end{enumerate}
then $\e_n = \profile_n \sharp \lambda$ converges to $\NEe = \NEprofile \sharp \lambda$ where $\NEe \in \mes_G(V)$ is the unique Nash equilibrium distribution.
\end{theorem}
The proof requires a few intermediate steps:
\begin{lemma}\label{Convtozero}
Consider a sequence of positive real numbers $\{ a_n \}_{n \in \N}$ such that $\sum_{n=1}^{\infty} \frac{a_n}{n} < +\infty$. Then we have
$$\lim_{N \to \infty} \frac{1}{N} \sum_{n=1}^{N} a_n = 0.$$
In addition, if there is a constant $C >0 $ such that $\abs a_n - a_{n+1} \abs < \frac{C}{n}$ then $\lim_{n \to \infty} a_n = 0$
\end{lemma}
\begin{proof}
See \cite{MondrerShapley2}.
\end{proof}
\begin{lemma}\label{ObviousLemma}
For every $y,z \in W^*$ and any $A\subseteq W$ we have :
$$\forall a\in Q_A (y): \quad h^*_A (y) - h^*_A (z) \leq \langle y - z , a \rangle .$$
\end{lemma}
\begin{proof}
The Lemma is obvious since $ h^*_A(y) - \langle y , a \rangle + h(a) = 0 \leq h^*_A(z) - \langle z , a \rangle + h(a)$.
\end{proof}
\begin{proof}[Proof of Theorem \ref{OMDNAG}]
Let $\NEprofile \in \mathcal{A}$ be a Nash equilibrium profile. Define the real sequence $\{\phi_n \}_{n\in \N}$ as follows:
$$\forall n\in \N: \quad \phi_n = \int_{I} \left( h (\NEprofile (i)) + h^*_{A_i} (\dual_n(i)) - \langle \dual_n(i) , \NEprofile(i)  \rangle \right) \; \d \lambda(i).$$
We first show $ \abs \phi_n \abs < \infty$. We have
$$ \int_{I} \left( h (\NEprofile (i)) + h^*_{A_i} (\dual_n(i)) - \langle \dual_n(i) , \NEprofile(i)  \rangle \right) \; \d \lambda(i) $$ $$=  \int_{I} \left( h (\NEprofile (i)) - h (\profile_n (i)) - \langle \dual_n(i) , \NEprofile(i) -\profile_n (i) \rangle \right) \; \d \lambda(i)$$
and
$$ \abs \; h (\NEprofile (i)) - h (\profile_n (i)) - \langle \dual_n(i) , \NEprofile(i) -\profile_n (i) \rangle \; \abs
\leq 2 R(M) + 2 \norm \dual_n (i) \norm_{W^*} M, $$
since $\norm \NEprofile (i) \norm_W , \norm \profile_n (i) \norm_W \leq M$ for $\lambda-$almost every $i\in I$. In addition by \eqref{OMD}$(i)$:
$$\forall n\in \N: \quad \norm \dual_n \norm_{\infty} \leq k(1 + \frac{1}{2} + \cdots + \frac{1}{n-1}) + \norm \dual_0 \norm_{\infty}$$
which proves $\abs \phi_n \abs < \infty$.
By definition of Fenchel conjugate we have $\phi_n \geq 0$.
Let us compute the difference $\phi_{n+1} - \phi_n$:
$$
\phi_{n+1} - \phi_{n} = \int_{I} \left( h^*_{A_i} (\dual_{n+1}(i)) - h^*_{A_i} (\dual_n(i)) - \langle \dual_{n+1}(i)-\dual_n(i) , \NEprofile (i)  \rangle \right) \; \d \lambda(i)
$$
So from Lemma \ref{ObviousLemma}:
\begin{equation*}
\begin{split}
\phi_{n+1} - \phi_{n} &\leq \int_{I}\langle \dual_{n+1}(i)-\dual_n(i) ,\profile_{n+1}(i) - \NEprofile(i)  \rangle \; \d \lambda(i) \\
&= - \beta_n\int_{I}\langle y (\profile_{n}(i) , \e_n) ,\profile_{n+1}(i) - \NEprofile (i)  \rangle \; \d \lambda(i) \\
&= - \beta_n\int_{I} \left( \langle y (\profile_{n}(i) , \e_n) ,\profile_{n}(i) - \NEprofile (i)  \rangle + \langle y (\profile_{n}(i) , \e_n) ,\profile_{n+1}(i) - \profile_{n}(i)  \rangle \right)  \; \d \lambda(i) \\
&\leq - \beta_n \alpha_n  + C\beta_n^2
\end{split}
\end{equation*}
where $\alpha_n = \int_{I}\langle y (\profile_n(i) , \e_n) ,\profile_n(i) - \NEprofile (i)  \rangle \; \d \lambda(i)$ and since by condition \eqref{OMDcondition} we have: 
$$\abs \langle y (\profile_{n}(i) , \e_n) ,\profile_{n+1}(i) - \profile_{n}(i)  \rangle \abs \leq k \norm \profile_{n+1}(i) - \profile_{n}(i) \norm_{W} $$
$$\leq \frac{k}{2K} \norm \dual_{n+1}(i) - \dual_{n}(i) \norm_{W^*} = \beta_n \frac{k}{2K} \norm y (\profile_{n}(i) , \e_n) \norm_{W^*} \leq \beta_n \frac{k^2}{2K}. $$
From to the definition of the subgradient we have:
$$\forall b \in W : \quad \langle y (a,\e_n) , a-b \rangle \geq \J (a,\e_n) - \J (b,\e_n) . $$
So:
$$\alpha_n = \int_{I}\langle y (\profile_n(i) , \e_n) ,\profile_n(i) - \NEprofile (i)  \rangle \; \d \lambda(i) \geq \int_{I} \left( \J (\profile_n(i) , \e_n) - \J (\NEprofile(i) , \e_n) \right) \; \d \lambda(i) =$$ $$ \int_{X} \J (a, \e_n)\; \d (\e_n - \NEe)(a) \geq \int_{X} \J (a, \NEe)\; \d (\e_n - \NEe)(a) = \psi_n \geq 0,$$
by Remark \ref{Col1}. Since $\beta_n = \frac{1}{n}$ we have:
\begin{equation}\label{SUM/n}
\sum_{n=1}^{N} \frac{\psi_n}{n} \leq \sum_{n=1}^{N} \frac{\alpha_n}{n}  = \sum_{n=1}^{N} \beta_n \alpha_n\leq \sum_{n=1}^{N} \left( \phi_n - \phi_{n+1} + \frac{C}{n^2} \right) = \phi_1 - \phi_{N+1} + \sum_{n=1}^{N} \frac{C}{n^2} < +\infty
\end{equation}
and:
$$\psi_{n+1} - \psi_ n = \int_{X} \J (a, \NEe)\; \d (\e_{n+1} - \e_{n})(a) = \int_{X} \left( \J (\profile_{n+1}(i), \NEe) - \J (\profile_{n}(i), \NEe) \right) \; \d \lambda(i).$$
From the definition of subgradient:
$$
\langle y(\profile_{n}(i) , \NEe) , \profile_{n+1}(i) - \profile_{n}(i) \rangle
\leq \J (\profile_{n+1}(i), \NEe) - \J (\profile_{n}(i), \NEe) \leq \langle y(\profile_{n+1}(i) , \NEe) , \profile_{n+1}(i) - \profile_{n}(i) \rangle $$
so $\abs \J (\profile_{n+1}(i), \NEe) - \J (\profile_{n}(i), \NEe) \abs = O(1/n)$ which gives $\abs \psi_{n+1} - \psi_n \abs = O(1/n)$. The later and inequality \eqref{SUM/n} yield $\lim_{n\to \infty} \psi_n=0$ by Lemma \ref{Convtozero}.

Since $\mes_G(V)$ is precompact, there exist a sequence $\{n_i \}_{i \in \N} \subseteq \N$ and $\e' \in \overline{\mes_G (V)}$ such that $\lim_{i \to \infty} \e_{n_i} = \e'$. Since $\J(\cdot , \NEe) : V \to \R$ is lower semi-continuous, we have:
$$
\int_V \J (a,\NEe) \; \d ( \e' - \NEe ) (a) \leq \lim \inf_{i} \int_V \J (a,\NEe) \; \d ( \e_{n_i} - \NEe )= \lim \inf_{i} \psi_{n_i} =0 , $$
which yields $\e' = \NEe$ due to the Corollary \ref{COR:DisintCOV} and the definition of Nash equilibrium distribution. So every accumulation point of set $\{ \e_n \}_{n\in \N} \subseteq \mes_G(V)$ is $\NEe$ which gives
$ \lim_{n \to \infty} \e_n = \NEe$
since $\mes_G (V)$ is precompact.
\end{proof}

\section{Application to First Order Mean Field Games}
\subsection{Model}
Let us define the first-order mean field games as special example of anonymous games proposed in section 1.  
Set $I=\R^d$ with the usual topology as the set of players and $m_0 \in \mes(I)$ a given Borel probability measure on $\R^d$. Let $V= \C^0 ([0,T] , \R^d)$ endowed with the supremum norm $\norm \y \norm_{\infty} = \sup_{t\in [0 , T ]} \norm \y(t) \norm$. For each player $ i \in \R^d $ let $A_i= S_{i,M}\subseteq V$ where:
\begin{equation}\label{M}
\forall x\in \R^d , \; M>0: \quad S_{x,M}:=  \{ \gamma \in H^1 ([0,T] , \R^d) \; | \;\gamma(0) = x , \; \norm \dot{\y} \norm_{L^2} \leq \sqrt{T} M \},
\end{equation}
i.e. the action set of every player $i\in\R^d$ is the paths with initial points equal to $i$ and bounded $L^2-$norm of velocity. We will explain how to choose $M>0$ properly.

Let $\mes(V)$ be the set of Borel probability measures on $V$ and set for every $t\in [0,T]$ the evaluation function $e_t : V \to \R^d$ as $e_t(\y) = \y(t)$. The MFG cost function $\J: V \times \mes(V) \to \R$ is defined as follows:
$$\J (\gamma , \eta) = 
\begin{cases}
\int_{0}^{T} \left(  L(\gamma_t , \dot{\gamma}_t) + f(\gamma_t , e_t \sharp \eta) \right) \; \d t + g(\gamma_T , e_T \sharp \eta), \quad &\text{if} \; \y \in H^1 ([0,T] , \R^d) \\
+\infty &\text{otherwise}.
\end{cases}
$$
It describes that a players choosing a path $\y$ has to pay first for its velocity, by the map Lagrangian $L: \R^d \times \R^d \to \R$, and second for his interaction with the population and congested areas, by the couplings $f,g : \R^d \times \mes(\R^d) \to \R$. We call the Anonymous Game
$$G=(\R^d, m_0 ,\C^0 ([0,T] , \R^d),(S_{i,M})_{i\in \R^d} , \J )$$
defined above, a \textit{first-order mean field game}. 

We suppose from now on, that the following conditions $(*)$ hold:
\begin{enumerate}[i.]
\item
$m_0$ has a compact support,
\item
for every $ x \in \R^d$ the function $L(x,\cdot) : \R^d \to \R^d$ is strongly convex $\C^2(\R^d , \R^d)$ function and there exists a non-decreasing coercive map $\theta : \R_+ \to \R_+$ such that: 
$$ \exists\; c_0 , C > 0: \quad  \theta(\norm v \norm) - c_0 \leq L(x,v) \leq C(1 + \norm v \norm^2),$$
$$\forall x\in B_r, v\in \R^d: \quad \max (\norm L_x(x,v) \norm , \norm L_v(x,v) \norm ) \leq C(r) \theta(\norm v \norm)$$
and the map $\theta$ satisfies the following growth condition:
$$\forall D>0, \exists K_D >0: \quad \theta(q+m) \leq K_D (1 + \theta(q)) \quad \text{for every } q\in \R_+ ,\; m\in[0,D]$$
\item
the couplings $f,g$ are conitnuous and for every $m \in \mes(\R^d)$ the maps $f(\cdot , m) , g(\cdot , m) : \R^d \to \R^d$ are $\C^1(\R^d , \R^d)$ and there exist $C,b>0$ such that:
$$\forall \; (x,m) \in \R^d \times \mes(\R^d): \quad \max(\norm f_x(x,m) \norm , \norm g_x(x,m) \norm ) \leq C (\norm x \norm^{b} + 1).$$
\end{enumerate}
The convexity of $L(x,\cdot)$ implies that for every $\eta \in \mes(V)$ the global cost $\J (\cdot , \e)$ is lower semicontinuous. It can be shown (see \cite{CannarsaSinestrari}, Section 6) if condition $(*)(ii,iii)$ hold, then there is at least one minimizer of
$$\min_{\y \in \mathcal{AC}([0,T] , \R^d), \; \y(0)=x}\int_{0}^{T} \left(  L(\gamma_t , \dot{\gamma}_t) + f(\gamma_t , e_t \sharp \e) \right) \; \d t + g(\gamma_T , e_T \e)$$
where $\mathcal{AC}([0,T] , \R^d)$ denotes the set absolutely continuous function from $[0,T]$ to $\R^d$. The minimizer $\y : [0,T] \to \R^d$ belongs to $\C^1 ([0,T] , \R^d)$ and moreover $L_v (\y_t , \dot{\y}_t)$ is absolutely continuous and $\y$ satisfies the Euler-Lagrange equation:
\begin{equation}\label{EulerLagrange}
\left\{
\begin{split}
(i) \qquad &\frac{\d}{\d t} L_v (\y_t , \dot{\y}_t) =L_x (\y_t , \dot{\y}_t) + f_x (\y_t , e_t \sharp \e), \quad \text{for almost every $t\in [0,T]$,} \\
(ii) \qquad &\y(0) = x, \quad \frac{\d}{\d t} L_v (\y_t , \dot{\y}_t) \abs_{t=T} = - g_x (\y_T , e_T \sharp \e).  \end{split}\right.
\end{equation}
In addition by $(*)(iii)$ there is $M(r)>0$ such that 
\begin{equation}\label{EULAGbounded}
\norm \y \norm_{\infty}, \norm \dot{\y} \norm_{\infty} \leq M(r)
\end{equation}
for every solution of \eqref{EulerLagrange} with $\y(0) = x$ and $x\in B_r$ (see \cite{CannarsaSinestrari} Theorem 6.3.1). 
The later ODE does not necessarily posses a unique solution. The uniqueness is more subtle. Define the value function $u: [0,T]\times \R^d \to \R$ as:
$$\forall (s,x) \in [0,T]\times \R^d: \quad u(s,x) = \min_{\y \in \mathcal{AC}([0,T] , \R^d) , \y(s)= x }  \int_{s}^{T} \left(  L(\gamma_t , \dot{\gamma}_t) + f(\gamma_t , e_t \sharp \e ) \right) \; \d t + g(\gamma_T , e_T \sharp \e).$$
Then it can be shown (see \cite{CDnotesonMFG}, section 4) that $u(0,\cdot) : \R^d \to \R$ is Lipschitz continuous and if it is derivable at $x\in \R^d$ then the solution of $\eqref{EulerLagrange}$ satisfies
\begin{equation}\label{UNEulerLagrange}
\left\{
\begin{split}
(i) \qquad &\frac{\d}{\d t} L_v (\y_t , \dot{\y}_t) = L_x (\y_t , \dot{\y}_t) + f_x (\y_t , e_t \sharp \e), \quad \text{for almost every $t\in [0,T]$, } \\
(ii) \qquad &\y_0 = x, \quad L_v (x , \dot{\y}_0) = - u_x (0,x).  \end{split}\right.
\end{equation}
It is a classical ODE with initial values and hence it possesses a unique solution. In other words, the optimal trajectories are unique in the initial points where $u(0,\cdot)$ is derivable i.e. almost everywhere since $u(0,\cdot)$ is Lipschitz. It provides the unique minimizer condition for first-order MFGs. Moreover, we choose $M$ in \eqref{M} equals to $M(R)$ in \eqref{EULAGbounded} with $$R=\sup_{x\in {\rm supp}(m_0)} \norm x \norm < +\infty.$$

\begin{remark}
If $K \subseteq \R^d$  be compact such that $\text{supp}(\m_0) \subseteq K$, then by Arzela-Ascoli the following set:
$$S = S_{K,M}=  \{ \gamma \in H^1 ([0,T] , \R^d) \; | \;\gamma(0) \in K \; , \; \norm \dot{\y} \norm_{L^2} \leq \sqrt{T}M \}.$$
is precomact in $V$. For every $\e \in \CVm$ we have $\text{supp}(\e) \subseteq S$, so $\CVm$ is tight and hence it is precompact in $(\mes_1 (V) , \d_1)$.	
\end{remark}

\begin{corollary}
The first-order MFGs defined above, satisfies the (\H) conditions and hence by Theorem \ref{THM:NashExist} has at least a Nash Equilibrium.
\end{corollary}
Under the conditions more restrictive than in $(*)$, the first-order MFG system defined by a coupled Hamilton-Jacobi (backward) and Fokker-Planck (forward) equation:
\begin{equation}\label{MFGintrtro}
	\left\{
	\begin{split}
		(i) \qquad &-\partial_t u +H(x,\nabla u(t,x)) =f(x,m(t)) \\
		(ii) \qquad &\partial_t m  - {\rm div} ( mD_p H (x,\nabla u)) =0  \\
		&m(0,x)=m_0(x) , \; u(T,x)=g(x, m (T)),
	\end{split}\right.
\end{equation}
has at least a weak solution i.e. there are $u : [0,T] \times \R^d \to \R , \; m \in L^\infty ([0,T] \times \R^d, \R) $ such that $u$ is the unique semiconvex viscosity solution of \eqref{MFGintrtro}$(i)$ and $m$ solves \eqref{MFGintrtro}$(ii)$ in distribution sense (see \cite{CDnotesonMFG}, \cite{LL06cr2}).
After solving the MFG system, the equilibrium $\NEprofile \in \A$ is obtained by:
$$\NEprofile (i) = \arg \min_{\y \in \mathcal{AC}([0,T] , \R^d)}  \int_{0}^{T} \left(  L(\gamma_t , \dot{\gamma}_t) + f(\gamma_t , m_t) \right) \; \d t + g(\gamma_T , m_T)
,$$ 
and in addition $ e_t \sharp \NEe = m_t$ for $\NEe = \NEprofile \sharp \m_0$. Lasry, Lions \cite{LL06cr2} have shown that MFG solution is unique if the couplings $f,g$ are strictly monotone i.e. for all $m \neq m' \in \mes(\R^d):$
$$
\int_{\R^d} (f(x,m)-f(x,m')) \d (m-m')(x)> 0, \qquad \int_{\R^d} (g(x,m)-g(x,m')) \d (m-m')(x)> 0
$$
We prove that the uniqueness is a consequence of the monotonicity of the cost function $\J$ and unique minimizer condition:
\begin{lemma}\label{MonotoneMFG}
If $f,g : \R^d \times \mes(\R^d) \to \R$ are monotone, then the MFG cost function will be so.
\end{lemma}
\begin{proof}
Let $\eta_1 , \eta_2 \in \mes(V)$. If we define $\m_{i,t} = e_t \sharp \eta_i$ for $i=1,2$ and $ t\in [0,T]$, we then have:
$$ \int_{V} \left(  \J(\gamma , \eta_1) - \J (\gamma , \eta_2) \right) \; \d (\eta_1 - \eta_2)(\gamma) =$$
$$ \int_{V} \left( \int_{0}^{T} \left( f(\gamma_t , \m_{1,t} ) - f(\gamma_t , \m_{2,t} ) \right)  \; \d t + g(\gamma_T , \m_{1,T} ) - g(\gamma_T , \m_{2,T} ) \right) \d (\eta_1 - \eta_2)(\gamma) = A+B$$
where
$$  A= \int_{0}^{T} \left( \int_{\R^d} \left( f(x , \m_{1,t}) - f(x , \m_{2,t})\right) \; \d (\m_{1,t} - \m_{2,t})(x) \right) \d t  \geq 0$$ $$B= \int_{\R^d} \left( g(x , \m_{1,T}) - g(x , \m_{2,T}\right) ) \; \d (\m_{1,T} - \m_{2,T})(x) \geq 0,$$
since the couplings $f,g$ are monotone.
\end{proof}
\begin{corollary}
The Monotone First Order MFG satisfying $(*)$  has at most one equilibrium. We have shown that there is at least one, so the game possesses a unique equilibrium.
\end{corollary}

\subsection{Fictitious Play in First Order MFG}
The fictitious play in first-order MFG takes such form: for initial profile of actions $$\profile_1 \in \A, \; \bar{\e}_1 = \e_1 = \profile_1 \sharp \lambda \in \mes(V)$$ the players play as follows for round $n=1,2,\ldots$ :
\begin{equation}\label{MFGFicitiousPlay}
\begin{array}{lrll}
(i)&\profile_{n+1}(i) &= &\arg\max_{\y \in H^1 , \y(0) = i} \int_{0}^{T} \left(  L(\gamma_t , \dot{\gamma}_t) + f(\gamma_t , e_t \sharp \bar{\e}_{n}) \right) \; \d t + g(\gamma_T , e_T \sharp \bar{\e}_{n}), \\
(ii)&\e_{n+1} &= &\profile_{n+1} \sharp \lambda ,\\
(iii)&\bar{\e}_{n+1} &= & \frac{1}{n+1} \sum_{i=1}^{n+1} \e_{i}.
\end{array}
\end{equation}
where $(i)$ holds for $m_0-$almost every $i\in \R^d$ . Here we apply the convergence result in fictitious play (Section 3) for monotone first-order MFG. We suppose the $(*)$ (and hence (\H)) conditions hold. 

\begin{lemma}\label{MFGFictiLem}
If $f,g: m \to f(\cdot , m) , g(\cdot , m)$ are Lipschitz from $\mes(\R^d)$ to $\mathcal{C}^1(\R^d)$ then there is a constant $C>0$ such that:
$$\abs
\J(\y,\e) - \J(\y,\e') - \J(\y',\e) + \J(\y',\e')
\abs \leq C \; \norm \y-\y' \norm_{\infty} \; \d_1 (\e , \e')$$
$$\abs
\J(\y,\e) - \J(\y,\e')
\abs \leq C \; \d_1 (\e , \e')$$
for every $\y,\y' \in H^1 ([0,T] , \R^d)$ and $\e,\e' \in \mes(V)$.
\end{lemma}
\begin{proof}
Since $f: m \to f(\cdot , m)$ is Lipschitz from $\mes(\R^d)$ to $\mathcal{C}^1 (\R^d)$ there is $C>0$ such that:
$$ \norm f(\cdot , m) - f(\cdot , m') \norm_{\C^1} \leq C \d_1 (m , m'), \quad \norm g(\cdot , m) - g(\cdot , m') \norm_{\C^1} \leq C \d_1 (m , m')$$
which means that for every $x,x'\in \R^d$ we have
$$\abs f(x , m) - f(x , m') - f(x' , m) + f(x' , m') \abs \leq C \norm x-x' \norm \d_1 (m , m'),$$
$$\abs f(x , m) - f(x , m') \abs \leq C \d_1 (m , m').$$
Similar inequalities hold with respect to $g$. We have:
$$
\abs \J(\y,\e) - \J(\y,\e') - \J(\y',\e) + \J(\y',\e') \abs$$ $$ \leq \int_{0}^{T} \abs f(\y(t) , e_t \sharp \e) - f(\y(t) , e_t \sharp \e') + f(\y'(t) , e_t \sharp \e) - f(\y'(t) , e_t \sharp \e') \abs \; \d t
$$
$$
+ \abs g(\y(T) , e_T \sharp \e) - g(\y(T) , e_T \sharp \e') - g(\y'(T) , e_T \sharp \e) + g(\y'(T) , e_T \sharp \e')\abs
$$
$$
\leq C \int_{0}^{T} \norm \y(t) - \y'(t) \norm \; \d_1(e_t \sharp \e , e_t \sharp \e') \; \d t+
\norm \y(T) - \y'(T) \norm \; \d_1(e_T \sharp \e , e_T \sharp \e')
$$
$$
\leq C \int_{0}^{T} \norm \y - \y' \norm_{\infty} \; \d_1(\e , \e') \; \d t + \norm \y - \y' \norm_{\infty} \; \d_1(\e , \e') = (CT +1) \; \norm \y - \y' \norm_{\infty} \; \d_1(\e , \e'),
$$
and 
$$
\abs \J(\y,\e) - \J(\y,\e') \abs \leq \int_{0}^{T} \abs f(\y(t) , e_t \sharp \e) - f(\y(t) , e_t \sharp \e') \abs \; \d t + \abs g(\y(T) , e_T \sharp \e) - g(\y(T) , e_T \sharp \e') \abs
$$
$$
\leq C \int_{0}^{T} \; \d_1(e_t \sharp \e , e_t \sharp \e') \; \d t+
\d_1(e_T \sharp \e , e_T \sharp \e')
\leq (CT +1) \; \d_1(\e , \e').$$
\end{proof}
\begin{corollary}
If $f,g: m \to f(\cdot , m) , g(\cdot , m)$ are Lipschitz, then by Lemma \ref{MFGFictiLem}, the convergence result of fictitious play (Theorem \ref{FictitiousPlayMonotone}) holds for the first-order montone MFG.
\end{corollary}

\subsection{Online Mirror Descent in First Order MFG}

Here we use the convergence result proved in Section 4 for the first-order MFG which have a monotone convex cost function $\J$. So let us suppose that the couplings $f,g$ are monotone and $L(\cdot , \cdot),f(\cdot , m),g(\cdot , m)$ are convex for every $m \in \mes(\R^d)$. It easily yields that $\J$ is monotone (by Lemma \ref{MonotoneMFG}) and for every $\e \in \mes(V)$, the function $\J(\cdot , \e) : H^1 ([0,T] , \R^d) \to \R$ is convex.

Let us set $W = H^1 ([0,T] , \R^d)$ endowed with inner product:
$$\forall \; \y_1, \y_2 \in W : \quad
\langle \y_1 , \y_2 \rangle_{W} = \langle \y_1(0), \y_2(0) \rangle_{\R^d} + \int_0^T \langle \dot{\y}_1(t) , \dot{\y}_2(t) \rangle_{\R^d} \; \d t. $$
We clearly have
$$ \bigcup_{i\in I} A_i \subseteq W \subseteq V, $$
and $A_i$ are uniformly bounded in $W$ for $m_0-$almost every $i\in I$.

For integrable functions $F , D \in L^2 ([0,T], \R )$ and $G \in \R$ we define $y = [[F,D,G]] \in W^*$ by:
$$\langle y , \y \rangle  = \int_{0}^{T} \left( F(t) \cdot \y_t + D(t) \cdot \dot{\y}_t \right)  \; \d t + G \cdot \gamma_T, \quad \text{for every } \y\in W $$
After a few computation we have:
$$\langle y , \y \rangle = \int_{0}^{T} \left( \int_{t}^{T} F(s) \; \d s +D(t) + G \right) \cdot \dot{\y}_t   \; \d t + \left( \int_{0}^{T} F(s) \; \d s + G \right) \cdot \gamma_0 = \langle \y_y , \y \rangle_{W}$$
where $\y_y \in W$ is the representation of $y\in W^*$ that solves:
\begin{equation}
\y_y(0) =\int_{0}^{T} F(s) \; \d s + G , \quad \frac{\d}{\d t} (\y_y) (t) =  \int_{t}^{T} F(s) \; \d s +D(t) + G.
\end{equation}
or
\begin{equation}\label{Reisz}
\y_y (t) = \int_{0}^{T}  F(s) \min(t,s) \; \d s + \int_{0}^{t} D(s) \; \d s +  (t+1)G + \int_{0}^{T} F(s) \; \d s. \end{equation}
If in addition to the condition $(*)(ii,iii)$ we have:
\begin{equation}\label{L_xL_v}
\forall \; (x,v) \in \R^d \times \R^d: \quad \norm L_v(x,v) \norm + \norm L_x(x,v) \norm \leq C ( \norm v \norm +1)
\end{equation}
then one can easily conclude by dominated Lebesgue convergence theorem that the function $\J (\cdot , \e) : W \to \R$ is differentiable for every $\e \in \mes(V)$. So the sub-differential set is singleton i.e. $ \partial \J (\cdot , \e)(\y) = \{ D_{\y} \J (\y,\e) \} \subseteq W^*$ and the derivative is calculated by:
$$\forall \; z\in W: \quad \langle D_{\gamma} \J (\gamma,\eta) , z \rangle = \lim_{\epsilon \to 0} \frac{\J (\gamma + \epsilon z , \eta) - \J(\gamma , \eta)}{\epsilon}$$
$$
= \int_{0}^{T} \left( L_x(\gamma_t,\dot{\gamma}_t)\cdot z_t + L_v(\gamma_t,\dot{\gamma}_t)\cdot \dot{z}_t + f_x (\gamma_t , e_t \sharp \eta)\cdot z_t \right) \; \d t + g_x(\gamma_T , e_T \sharp \eta)\cdot z_T
$$
or according to our representation:
$$D_{\y} \J (\y,\e) = [[ L_x(\y_{(\cdot)},\dot{\y}_{(\cdot)}) + f_x (\y_{(\cdot)} , \; e_{(\cdot)} \sharp \e) , L_v (\y_{(\cdot)},\dot{\y}_{(\cdot)}) , \; g_x (\y_T , e_T \sharp \e) ]].$$
So by the computation in \eqref{Reisz} the gradient $\nabla_{\y} \J (\y,\e)\in W$ is obtained as follows:
\begin{equation}
\begin{split}
\nabla_{\y} \J (\y,\e) (t) &= 
\int_{0}^{T} \left( L_x(\y_{s},\dot{\y}_{s}) + f_x (\y_{s} , e_{s} \sharp \e) \right) \min(t,s) \; \d s + \int_{0}^{t} L_v(\y_{s},\dot{\y}_{s}) \; \d s \\
&+ (t+1)g_x(\y_T , e_T \sharp \e) + \int_{0}^{T} \left( L_x(\y_{s},\dot{\y}_{s}) + f_x (\y_{s} , e_{s} \sharp \e) \right) \; \d s.
\end{split}
\end{equation}
\begin{theorem}\label{OMDMFG}
Suppose a first-order MFG satisfies the $(*),\eqref{L_xL_v}$ conditions. If the cost function $\J$ is monotone and convex w.r.t. first argument, then the online mirror descent algorithm proposed in \eqref{OMD} for $h : W \to \R, \; h(\y) = \frac{1}{2} \norm \y \norm^2_{H^1}$ and $\beta_n= \frac{1}{n}\; (n \in ÷N)$, converges to the unique first-order mean field game equilibrium.

\end{theorem}
\begin{proof}
The function $h : W \to \R, \; h(\y) = \frac{1}{2} \norm \y \norm^2_{H^1}$ is $\frac{1}{2}-$strongly convex function and lower semicontinuous for the weak topology, so the mirror projection $Q_{A_i}$ will have singleton values. 

The game satisfies the (\H) conditions. By $(*)(ii)$ the derivatives $f_x(\y_t , m),g_x(\y_t,m)$ are bounded by a power of $\norm \y \norm_{H^1}$. Then since the condition \eqref{L_xL_v} holds, there are $C' , \alpha >0$ such that for $\lambda-$almost every $i\in I$:
$$ \max ( \norm D_\y \J (\profile_{n}(i) , \NEe) \norm_{W^*} , \norm D_\y \J (\profile_{n}(i) , \e_n)  \norm_{W^*} ) \leq C' ( \norm \dot{\profile}_{n}(i) \norm_{L^2 ([0,T] , \R^d)}^{\alpha} + 1) \leq C'((\sqrt{T}M)^\alpha +1).$$
So all of the conditions in Theorem \ref{OMDNAG} are satisfied and the desired convergence result holds.
\end{proof}
\begin{remark}
Since the space $H^1([0,T] , \R^d)$ is Hlibert, we identify it by its dual space. Hence by choice $h(\y) = \frac{1}{2} \norm \y \norm^2_{H^1}$ we have:
$$Q_{A_i} (\y) = \pi_{A_i}(\y) = \frac{\min (\norm \dot{\y} \norm_{L^2} , \sqrt{T} M) }{\norm \dot{\y} \norm_{L^2} }(\y - \y_0) + i .$$
by the choice of $A_i$. Then, the OMD algorithm have such form
\begin{equation}\label{OMDMFG}
\begin{array}{lllll}
(i)&\dual_{n+1}(i)&= &\dual_n (i) - \frac{1}{n} \nabla \J (\profile_n (i) , \e_n), &\text{for every} \; i \in I \\
(ii)&\profile_{n+1}(i) &= &\frac{\min (\norm \dot{\dual}_{n+1}(i) \norm_{L^2} , \sqrt{T} M)}{\norm \dot{\dual}_{n+1}(i) \norm_{L^2} }(\dual_{n+1}(i) - \dual_{n+1}(i)_0) + i  , &\text{for every} \; i \in I \\
(iii)&\e_{n+1} &= &\profile_{n+1} \sharp \lambda.&
\end{array}
\end{equation}
or in explicit way it takes the followin form: let $\hat{\y}_{0,x} = 0$ for every $x\in \R^d$ and:
\begin{equation}
\begin{split}
\hat{\y}_{n+1 , x} (t) &= \hat{\y}_{n,x}(t) - \frac{1}{n} \int_{0}^{T} \left( L_x(\y_{n , x} (s),\dot{\y}_{n , x} (s)) + f_x (\y_{n , x} (s) , e_{s} \sharp \e_n) \right) \min(t,s) \; \d s \\ 
&- \frac{1}{n} \int_{0}^{t} L_v (\y_{n , x} (s),\dot{\y}_{n , x} (s)) \; \d s
- \frac{t}{n} g_x(\y_{n , x} (T) , e_T \sharp \e_n), \\
\y_{n+1 , x} &= c_{n+1} \hat{\y}_{n+1,x} + x, \quad c_{n+1} = \frac{\min (\norm \dot{\hat{\y}}_{n+1,x} \norm_{L^2} , \sqrt{T} M)}{\norm \dot{\hat{\y}}_{n+1,x} \norm_{L^2} }, \\
\e_{n+1} &= \y_{n+1 , \cdot} \sharp \lambda.
\end{split}
\end{equation}
\end{remark}

\section{Appendix}
Here we demonstrate some disintegration lemmas used in the precedent proofs. Suppose $I$ a Polish space and $V$ a metric space. Let $A: I \to V$ be a correspondence with  $A(i)=A_i$. For a Borel probability measure $\lambda \in \mes(I)$ we say $\e \in \mes (V)$ \textit{disintegrates} with respect to $(A_i)_{i\in I}$ if there are $\{\e^i\}_{i\in I} \subset \mes (V) $ such that
$$\text{for $\lambda-$almost every } i\in I: \quad \text{supp}(\e^i) \subseteq A_i,$$
$$\text{for every bounded measurable } f:V\to \R: \quad \int_V f(a) \; \d \e (a) = \int_I \int_V f(a) \; \d \e^i (a) \; \d \lambda (i).$$

\begin{theorem}\label{THM:Disint}
Suppose $A: I \to V$ be upper semi continuous. Let $\{\e_n\}_{n\in \N} \subseteq \mes_1 (V)$ with $\e_n \to \e$ in weak sense. If for every $n\in \N$, $\e_n$ disintegrates with respect to $(A_i)_{i\in I}$ then the same holds true for $\e$.
\end{theorem}

\begin{proof}
For every $n \in \N$, define $m_n \in \mes(I \times V)$ as follows:
$$\text{for every bounded measurable } f:I \times V\to \R: \quad \int_{I \times V} f(i,a) \; \d m_n (i,a) = \int_I \int_V f(i,a) \; \d \e^i_n (a) \; \d \lambda (i).$$
Obviously $\pi_I \sharp m_n = \lambda, \pi_V \sharp m_n = \e_n$ where $\pi_I , \pi_V$ are respectively projections of $I \times V$ on $I,V$. Since $\{ \e_n \}$ are tight and $I$ is a Polish space, for every $\epsilon > 0$, there is a compact set $I_\epsilon \subseteq I , K_\epsilon \subseteq V$ such that $\lambda(I \setminus I_\epsilon) , \e_n (V \setminus K_\epsilon) < \epsilon$ for all $n \in \N$. In addition
$$m_n(I_\epsilon \times K_\epsilon) \geq 1 - m_n(I \times V \setminus K_\epsilon) - m_n(I \setminus I_\epsilon \times V) = 1 - \e_n( V \setminus K_\epsilon) - \lambda (I \setminus I_\epsilon) \geq 1 - 2 \epsilon,$$
which means the set $\{ m_n \}_{n\in \N}$ is tight too. Hence there exists $m \in \mes(I \times V)$ and a subsequence $\{ m_{n_k} \}_{k\in \N}$ such that $m_{n_k} \to m$. We directly have $\e_{n_k} = \pi_V \sharp m_{n_k} \to \pi_V \sharp m$ which means $\pi_V \sharp m = \e$. On the other hand, due to the disintegration theorem (see \cite{AGS} page 121) there are $m^i \in \mes(V)$ for every $i\in I$, such that
$$\text{for every bounded measurable } f:I \times V\to \R: \quad \int_{I \times V} f(i,a) \; \d m (i,a) = \int_I \int_V f(i,a) \; \d \m^i (a) \; \d \lambda (i).$$
So since the second marginal of $m$ is $\e$, we can write:
$$\text{for every bounded measurable } f: V\to \R: \quad \int_V f(a) \; \d \e (a) = \int_I \int_V f(a) \; \d \m^i (a) \; \d \lambda (i).$$
So what is left is to show that for $\lambda-$almost every $ i\in I$ we have $\text{supp}(m^i) \subseteq A_i$. Set $f: I \times V \to \R$ as $f(i,a) = 1_{a\in A_i}$. We know the function $f$ is upper semi continuous since the correspondence $A: I \to V,\; A(i)=A_i$ is upper semi continuous. For every $n\in \N$ we have:
$$\int_{I \times V} f(i,a) \; \d m_n (i,a)= \int_I \int_V f(i,a) \; \d \e^i_n (a) \; \d \lambda (i) = 1.$$
Hence
$$1 = \limsup_{k} \int_{I \times V} f(i,a) \; \d m_{n_k} (i,a) \leq \int_{I \times V} f(i,a) \; \d m (i,a) \leq 1,$$
so $\int_{I \times V} f(i,a) \; \d m (i,a)=1$ which is equivalent to say for $\lambda-$almost every $ i\in I$ we have $\text{supp}(m^i) \subseteq A_i$.
\end{proof}
\begin{corollary}\label{COR:DisintCOV}
Every element  $\e \in \CV$ disintegrates with respect to $(A_i)_{i\in I}, \lambda \in \mes(I)$.
\end{corollary}
\begin{proof}
Let $\Sc \subset \mes(V)$ be the set of all measures which disintegrates with respect to $(A_i)_{i\in I}$. Clearly $\Sc$ is convex and due to Theorem \ref{THM:Disint} it is closed. Also, we have $\mes_G (V) \subseteq \Sc$ since for all $\profile \in \A$ we have
$$\text{for every bounded measurable } f:V\to \R: \quad \int_V f(a) \; \d (\profile \sharp \lambda) (a) = \int_I \int_V f(a) \; \d \delta_{\profile(i)} (a) \; \d \lambda (i),$$
hence it gives $\CV \subseteq \Sc$. 
\end{proof}

\paragraph{Acknowledgement.} I would very much like to thank Panayotis Mertikopoulos, Sylvain Sorin and my Phd directors Pierre Cardaliaguet and Rida Laraki, who provided insight and expertise that greatly assisted this research as well as comments that improved the manuscript. However, the responsibility of the materials are completely by the author. In addition, the author was partially supported by the ANR (Agence Nationale
de la Recherche) projects ANR-14-ACHN-0030-01 and ANR-16-CE40-0015-01.

\end{document}